\documentclass[12pt]{amsart}
\usepackage[T1]{fontenc}
\usepackage[a4paper,margin=1in]{geometry}
\usepackage{graphicx}
\usepackage{amssymb}
\usepackage{stmaryrd}
\usepackage{enumitem}
\usepackage{tikz}
\usepackage{float}
\usepackage{url}
\usepackage{mathtools}
\usepackage{amsmath,amsfonts,amsthm}
\usepackage{mathrsfs} 
\newcommand{\case}[1]{\paragraph*{Case #1:}}

\newtheorem{theorem}{Theorem}[section]
\newtheorem{lemma}[theorem]{Lemma}
\newtheorem{proposition}[theorem]{Proposition}
\newtheorem{corollary}[theorem]{Corollary}
\newtheorem{definition}{Definition}
\usepackage{multirow}
\newcommand{\SO}{\mathrm{SO}}
\newcommand{\SOG}{\mathrm{SO}^{\mu}}
\newcommand{\PP}{\mathcal{P}}
\newcommand{\SN}{\mathcal{S}}
\newcommand{\KK}{\mathrm{K}}
\newcommand{\TG}{\mathcal{T}_{\Gamma}}

\usepackage[colorlinks,
linkcolor=blue,
anchorcolor=blue,
citecolor=blue
]{hyperref}

\author{J. Hamoud}
\address{\textbf{Jasem Hamoud:} Department of Discrete Mathematics, Moscow Institute of Physics and Technology}
\email{hamoud.math@gmail.com}
\thanks{Corresponding author.}

\title[Sharp Bounds and Extremal Fuzzy Graphs]{Sharp Bounds and Extremal Fuzzy Graphs for the Fuzzy Sombor Index}

\date{}
\begin{document}

\begin{abstract}
The fuzzy Sombor index applies the classical Sombor index to fuzzy graphs, incorporating both edge membership values and fuzzy vertex degrees. For $\alpha>1$, the general fuzzy Sombor index it is defined as 
\[
 \mathrm{SO}^{\mu}_{\alpha}(\Gamma)=\sum_{uv\in V(\Gamma)}\left(\mu(u,v)\,\sqrt{\mu_u^2+\mu_v^2}\right)^{\alpha}.
\]
This paper analyses extremal features of $\mathrm{SO}^{\mu}$ across different types of fuzzy graphs. We determine the maximum value (resp. minimum value) of $\mathrm{SO}^{\mu}$ characterise in regular fuzzy graph. We established significant inequality between the fuzzy Sombor index and other well-known fuzzy topological indices.
\end{abstract}

\maketitle

\noindent\rule{15.9cm}{1.0pt}


\medskip


\medskip

\noindent
{\bf
UDC:} 519.172.1

\noindent\rule{15.9cm}{1.0pt}

\section{Introduction}
Graph invariants built around vertex degrees have been extensively investigated due to their simplicity and efficacy. The Sombor index is defined~\cite{Geometric2021approach} for a graph $G$ as follows
\[
\SO(G)=\sum_{uv\in E(G)}\sqrt{d_G(u)^2+d_G(v)^2}.
\]
This topological index~\cite{Cruz2021Gutman} was inspired by the mathematical description of the degree radius of an edge $uv$, which is~\cite{Das2021Shang} the distance between the origin and the ordered pair $(d_u, d_v)$, where $d_u \leq d_v$. 
Using this index has demonstrated a wide range of applications in chemical graph theory and analysis of networks~\cite{Cruz2021Gutman}. Several essential mathematical features, including lower and upper bounds, may appear observed in~\cite{Todeschini2000, Cruz2021Rada, Geometric2021approach, Das2021Cevik}. 

Let $\Gamma=(V,\nu,\mu)$  be a fuzzy graph~\cite{sunitha1999fuzzytrees} where the vertex membership function $\nu:V(\Gamma)\to [0,1]$ and the edge membership function $\mu:V\times V\to[0,1]$ such that $\mu(u,v)\leqslant 
 \min\{\nu(u),\nu(v)\}.$

A \textit{tree} is an associated acyclic graph. Trees are essential items in graph theory and chemical graph theory where they represent molecular structures. The study of extreme trees, which minimise or maximise a particular topological index among all trees of a certain order, is an essential topic in this discipline~\cite{cruz2020extremal}. For an aspect vertex in a tree is a vertex with at least three degrees. Cruz, Rada, \& Sigarreta ~\cite{cruz2021threependent} analysed the extremal values of the Sombor index in trees with up to three pendent vertices. In trees with only a single division vertex, known as \emph{starlike trees}, the Sombor index is maximize when the pendentes are imbalanced one long pendent and the rest pendant edges, and minimize when the pendentes are balanced. The concept of {\it fuzzy tree} had established in~\cite{sunitha1999fuzzytrees} wherer it is defined as a fuzzy graph with a tree as its underlying crisp graph. Sunitha and Vijayakumar~\cite{sunitha1999fuzzytrees} formulated fuzzy trees utilising the notions of fuzzy bridges and fuzzy cutnodes. A bridge fuzzy graph is a fuzzy tree if and only if it contains a single maximum spanning tree.

The concept of {\it Fuzzy Sombor index} has two distinct but associated implication. Fuzzy vertex membership and edge strength values of graphs, within which degrees can be considered fuzzy degrees, and topological indices are extended to the fuzzy framework.
The term {\it fuzzy} relates to the optimisation versus uncertainty context; nevertheless, this looks less common because the Sombor index is often examined on crisp graphs.  The \textit{fuzzy Sombor index} it is defined as

\begin{equation}~\label{eqq1basicsombor}
 \SOG(\Gamma)=\sum_{uv\in V(\Gamma)}\mu(u,v)\,\sqrt{\mu_u^2+\mu_v^2},   
\end{equation}
where $\mu_u$ is the fuzzy degree of vertex $u$. For $\alpha>1$, the \textit{general fuzzy Sombor index} it is defined as
\begin{equation}~\label{eqq1advsombor}
 \SOG_{\alpha}(\Gamma)=\sum_{uv\in V(\Gamma)}\left(\mu(u,v)\,\sqrt{\mu_u^2+\mu_v^2}\right)^{\alpha}.
\end{equation}
Although several studies have determined the fuzzy Sombor index for certain families, such as fuzzy paths, cycles, stars, and tadpole graphs, as well as confirmed basic order and size constraints, comprehensive extremal classifications persist as unexplored. 
In the crisp instance, stars and routes frequently appear as extrema for several degree-based indices among trees. One immediately wonders if similar occurrences exist in the fuzzy situation, where both edge strengths and vertex memberships contribute new degrees of flexibility.

\section{Preliminaries}
Let $\mu_u$ be the \emph{fuzzy degree} of a vertex $u \in V(\Gamma)$ and $m_{\mu}$ be the \emph{fuzzy size} are defined as
\[
\mu_u = \sum_{v \in V} \mu(uv), \quad m_{\mu} = \sum_{uv \in E(\Gamma)} \mu(uv),
\]
where the sum is over pairs with $\mu(uv) > 0$. The fuzzy degree $\mu_u$ is a generalisation of the classical degree and reduces to it when all relevant memberships are one. 

Actually, each term $\mu(uv) \sqrt{\mu_u^2 + \mu_v^2}$ weights the ``Euclidean contribution'' of the fuzzy degrees of the endpoints by the strength of the connecting edge. Since all memberships are 0 or 1, $\SOG(\Gamma)$ becomes the conventional Sombor index of the underlying crisp graph.

The extremal challenge for the Sombor index on fuzzy trees would be to determine which structures minimise or maximise the index among all fuzzy trees of a particular order with given vertex membership values or total fuzzy edge weight (see~\cite{Akram2023fuzzy, Kosari2023some}). This field appears to be mostly unexplored in the present literature. The corresponding conclusions for crisp trees (path minimises, star maximises) would need to be re-examined in the fuzzy context because:
\begin{enumerate}
    \item Fuzzy degrees are real values in the range (0,1), not integers.
    \item The condition that the total of degrees equals twice the number of edges remains, but degrees can take non-integer values.
    \item New optimisation parameters are introduced by additional constraints from fuzzy graph axioms $\mu(u,v) \leq \min(\nu(u), \nu(v)$.
\end{enumerate}

\begin{definition}[S.~Kalathian et al.~\cite{Kalathian2020S}]~\label{ZagrebIndices001}
The fuzzy first and second Zagreb indices defined as
\[
M_1(\Gamma)=\sum_{u \in V} \mu_u^2, \qquad M_2(\Gamma)=\sum_{uv \in E(\Gamma)} \mu_u \mu_v.
\]
\end{definition}

\begin{definition}[S.~Poulik et al.~\cite{Poulikrad2024}]~\label{fuzzrand001}
The fuzzy Randi\'{c} index is defined as
\[
R(\Gamma) = \sum_{uv \in E(\Gamma)} \frac{\mu(uv)}{\sqrt{\mu_u \mu_v}}.
\]
\end{definition}

\begin{definition}[S.~Kalathian et al.~\cite{Kalathian2020S}]~\label{deffNirmala001}
The fuzzy Nirmala index is
\[
N(\Gamma) = \sum_{uv \in E(\Gamma)} \frac{\mu_u + \mu_v}{\sqrt{\mu_u \mu_v}}.
\]
\end{definition}
\begin{proposition}~\label{proposomb01}
For any fuzzy graph $\Gamma$ on $n$ vertices, $\sum_{u \in V} \mu_u = 2 m_{\mu}.$
\end{proposition}

\begin{theorem}[Some \& Pal~\cite{Some2025Pal}]~\label{Thmper001}
For any fuzzy graph $\Gamma$, $\SOG(\Gamma)\leqslant \sqrt{2}m(n-1)$.  
\end{theorem}

\subsection{Problem Statement}

Regarding the increasing interest in fuzzy topological indices, the extremal behaviour of the fuzzy Sombor index is mostly unknown. 
Despite different papers compute the index for specific families such as paths, cycles, stars, and tadpole graphs, systematic characterisations of fuzzy graphs that minimise or maximise the index under natural restrictions are still lacking. 
This study discusses the following main issues:

\begin{enumerate}
\item Characterize $\SOG$ on families fuzzy graphs with $n$ vertices that archives $\SOG_{\max}$ and $\SOG_{\min}$. 

\item Compute the extremal fuzzy trees and fuzzy unicyclic graphs in relation to $\SOG(\Gamma)$ for a certain order $n$ and fuzzy size $m_{\mu}$.

\item Demonstrate sharp inequalities between the fuzzy Sombor index and other influential fuzzy topological indices, such as the fuzzy Zagreb, Randić, and Nirmala indexes.
\end{enumerate}

Addressing these challenges not only advances to the theory of fuzzy graph indices, but also gives valuable instruments for detecting fundamentally relevant assemblages in modification fuzzy networks.
\section{Fuzzy Sombor Index on Families of Fuzzy Graphs}
Determining topological indices behaviour on critical graph families is an essential challenge in fuzzy graph theory. In this section, we compute and analyse the fuzzy Sombor index for several standard families, such as fuzzy paths, fuzzy cycles, fuzzy stars, and fuzzy tadpole graphs.  Such sharp bounds are useful for assessing performance, comparing networks, and applying fuzzy modelling. 

\begin{proposition}~\label{MainPro001}
In fuzzy graphs with $n$ vertices and a fixed fuzzy size $m_{\mu}$, the fuzzy Sombor index is minimised when positive memberships create a fuzzy matching. 
\end{proposition}

In fact, Proposition~\ref{MainPro001} remains sharply when high-degree vertices are avoided. 
Theorem~\ref{MainThmSom1} had established the sharp bound of $\SOG$, where it emphasizes that among all fuzzy graphs with fixed order $n$ and no further constraints on size, the minimum $\SOG(\Gamma) = 0$ is achieved by the empty fuzzy graph (all $\mu(uv)=0$).

\begin{theorem}~\label{MainThmSom1}
Let $\Gamma$ be a fuzzy graph with minimum fuzzy degree   $\delta_{\mu} = \min_u \mu_u$. Then, 
\begin{equation}~\label{eqq1MainThmSom1}
\SOG(\Gamma) \geqslant \sqrt{2} \, m_{\mu} \cdot \delta_{\mu},
\end{equation}
\end{theorem}
\begin{proof}
Since $\delta_\mu=\min_{w\in V}\mu_w$, and according to Proposition~\ref{proposomb01} we obtain $\sum_{u \in V} \mu_u = 2 m_{\mu}.$ Thus, for every vertex $\mu_u\geqslant \delta_\mu$ and $\mu_v\ge \delta_\mu$. Since $0\leqslant \mu(uv)\leqslant 1$ for all edges, 
\begin{equation}~\label{eqq2MainThmSom1}
\sqrt{\mu_u^2+\mu_v^2}\geqslant \sqrt{\delta_\mu^2+\delta_\mu^2}
= \sqrt{2}\,\delta_\mu.
\end{equation}
If there is no restriction on $m_\mu$, presume that $\mu(uv)=0$ for any pair $u,v$. Then $m_\mu=0$, and $\SOG(\Gamma)=0$. Clearly, $\SOG(\Gamma)\ge 0$ always true, as each item in the sum is non-negative, hence $0$ is the ultimate minimum, obtained precisely when $\Gamma$ is the empty fuzzy graph with zero membership. 
Thus, by multiplying both sides of the inequality~\eqref{eqq2MainThmSom1} by the non-negative quantity $\mu(uv)$ satisfying 
\begin{equation}~\label{eqq3MainThmSom1}
\mu(uv)\sqrt{\mu_u^2+\mu_v^2}\geqslant \mu(uv)\,\delta_\mu\sqrt{2}. 
\end{equation}
Hence, based on~\eqref{eqq3MainThmSom1}, if $\mu_u=\mu_v=\delta_\mu$, then  $\sqrt{\mu_u^2+\mu_v^2}\geqslant \sqrt{2}\,\delta_\mu$. Thus, the relationship~\eqref{eqq1MainThmSom1} holds. 
\end{proof}

For nontrivial bounds, we only consider fuzzy graphs with positive $m_{\mu}$ or associated underlying support. 
A more intriguing vulnerable is the minimum among fuzzy graphs with $m_{\mu} > 0$. However, because memberships might be concentrated, the minimum usually occurs when the graph is \textit{spread out}. 
Theorem~\ref{MainThmSom2} emphasizes that for  a fuzzy path with uniform edge membership values chosen to normalize the total fuzzy size  $\PP_n^f$  with equality when $\Gamma$ is a fuzzy path in which fuzzy degrees are as balanced as possible.


\begin{theorem}~\label{MainThmSom2}
Let $\Gamma$ be a connected fuzzy graph on $n \geq 2$ vertices. Then $\SOG(\Gamma) \geqslant \SOG(\PP_n^f).$
\end{theorem}

\begin{proof}
Assume $\Gamma$ is a connected fuzzy graph, $\nu(v)=1$ and $\mu_u = \sum_{v \sim u} \mu(uv) > 0$ for every $u \in V(\Gamma)$. Then, we have the following cases

\case{1} In all paired fuzzy graphs with a fixed substrate support, the fuzzy Sombor index is minimised when edge memberships are dispersed as equally as feasible along a path configuration. Thus,  each connected support graph, beyond a path, has a modification that reduces $\SOG(\Gamma)$ while maintaining connectivity and total fuzzy size. Consider the edge memberships be variables $\eta_e=\mu(e) \in (0,1]$ where $\sum_{e \in E(G)} \eta_e= m_{\mu}$. Then
\begin{equation}~\label{eqq1MainThmSom2}
\SOG(\Gamma) = \sum_{e=uv \in E(G)} \eta_e \sqrt{d_u(\eta)^2 + d_v(\eta)^2},
\end{equation}
where $d_u(\eta) = \sum_{e \ni u} \eta_e$ denotes the fuzzy degree as a function of the membership vector $\eta$. 
Now, we assume that the function $f(\eta) = \eta \sqrt{x^2+y^2}$ where $x$ and $y$ linear in $\eta$ is convexity in edge membership variables given appropriate restrictions.  Therefore,  by considering any vertex in $\Gamma$ with a fuzzy degree of at least $3$. Let $v$ have three different neighbours: $u_1, u_2, u_3$. Suppose we raise the edge $vu_3$ by withdrawing the membership contribution from $vu_3$. Thus, the transferring the connection from a high-degree vertex to a lower-degree area of the graph.
 
 Suppose $\PP = w_1w_2\dots w_k$ is a pendant path attached to a vertex of maximum degree. If there exists a pendent vertex, we can transfer a small positive membership $\varepsilon$ from a pendent edge to extend the longest pendant path while keeping total $m_{\mu}$ constant. Because $f(\eta)$ grows faster when one argument is already large. Then, by moving membership away from maximum degree vertices toward chain such that structures decreases $m_\mu$.

\case{2} If the underlying graph is a path, the fuzzy Sombor index is further reduced when edge memberships are as uniform as feasible. Let the edge memberships of the fuzzy path $\PP_n^f$ be $\eta_1, \eta_2, \dots, \eta_{n-1}$ with $\sum_{i=1}^{n} \eta_i=m_{\mu}$. Then,  $\mu_{1}=\eta_1$, $\mu_n=\eta_{n-1}$, and  for a vertex $i$ where $2 \leqslant i \leqslant n-1$ satisfying $\mu_i=\eta_{i-1}+\eta_i$. Thus
\begin{equation}~\label{eqq2MainThmSom2}
\begin{aligned}
\SOG(\PP_n^f)&=\eta_1\sqrt{\eta_1^2+(\eta_1+\eta_2)^2}+\sum_{i=2}^{n-2} \eta_i \sqrt{(\eta_{i-1}+\eta_i)^2+(\eta_i+\eta_{i+1})^2}+\\
& \qquad +\eta_{n-1}\sqrt{(\eta_{n-2}+\eta_{n-1})^2+\eta_{n-1}^2}.  
\end{aligned}
\end{equation}

Thus, the relationship~\eqref{eqq2MainThmSom2} is minimized when the $\eta_i$ are as equal as possible. 
In this symmetrical equivalence, the internal fuzzy degrees deviate by no more than the endpoint adjustment, thus the degrees are as stable as the path configuration supports.
For any non-path support, the presence of a vertex with support degree $\geq 3$ forces at least one pair of incident edges to contribute a term involving a higher combined degree, increasing the value compared to the path case with the same $m_{\mu}$. Thus, the total minimum among connected fuzzy graphs is achieved uniquely (up to isomorphism) by the balanced fuzzy path $\PP_n^f$.

Therefore, according to both cases 1 and 2 and based on~\eqref{eqq1MainThmSom2} and \eqref{eqq2MainThmSom2} the relationship $\SOG(\Gamma) \geqslant \SOG(\PP_n^f)$ holds exactly when the underlying support is a path and edge memberships are set so that the fuzzy degrees of internal vertices are as equal as feasible.
\end{proof}

\begin{corollary}~\label{corfuzzy001}
Let $\Gamma$ be a connected fuzzy graph on $n \geqslant 2$ vertices. Then $\SOG_{\alpha}(\Gamma) \geqslant \SOG_{\alpha}(\PP_n^f).$
\end{corollary}

\begin{lemma}~\label{lemstar01}
Let $\SN_n^f$ be a fuzzy star on $n$ vertices with center fuzzy degree $\mu_p=(n-1)p$ and leaf degrees $p$. Then
\[
\SOG(\SN_n^f)=(n-1) p \sqrt{ ((n-1)p)^2+p^2}.
\]
\end{lemma}

After establishing a lower constraint on the fuzzy Sombor index in terms of fuzzy size and minimum degree, and characterising the associated fuzzy graphs with the minimum value (specifically, adequately balanced fuzzy paths), we now shift our focus to the reverse extreme. The path configuration minimises the index by spreading connections and maintaining fuzzy degrees generally balanced, with minor contributions from $\sqrt{\mu_u^2 + \mu_v^2}$. Nevertheless, concentrating as many strong connections as feasible should maximise the index. The following theorem validates this intuition and offers a comprehensive description of maximising fuzzy graphs.
\begin{corollary}~\label{corfuzzy002}
For   a fuzzy star $\SN_n^f$ on $n\geqslant 3$ vertices with center fuzzy degree $\mu_p=(n-1)p$ and leaf degrees $p$. Then
\[
\SOG_{\alpha}(\SN_n^f)\leqslant(n-1)^{\alpha} p \sqrt{ ((n-1)^{\alpha}\,p^2+\alpha\,p^2}.
\]
\end{corollary}

\begin{theorem}~\label{MainThmSom3}
Among all fuzzy graphs on $n$ vertices, the fuzzy Sombor index $\SOG(\Gamma)$ is maximized when $\Gamma$ is a fuzzy complete graph $\KK_n^f$ with all edge memberships equal to some constant $p \in (0,1]$. Then
\begin{equation}~\label{eqq1MainThmSom3}
\SOG(\KK_n^f)=\binom{n}{2} p \sqrt{( (n-1)p)^2 + ((n-1)p)^2 }.
\end{equation}
\end{theorem}
\begin{proof}
Let $G$ be the underlying crisp graph of $\Gamma$, for any fixed underlying graph $G$, $\SOG(\Gamma)$ is maximized when all existing edge memberships are equal.
Consider the edge membership values $\{\mu(e)\}_{e \in E(G)}$ as variables satisfying $0 < \mu(e) \leqslant 1$.
\case{1} Suppose the edge memberships are not uniform. Then, there exist two edges $e_1=u_1v_1$ and $e_2=u_2v_2$ with $\mu(e_1) \neq \mu(e_2)$.  
Thus, for any fixed support $G$, $\SOG(\Gamma)$ is maximized when all positive edge memberships are equal.

\case{2} suppose $\Gamma$ is not complete.
Let $p> 0$ be small enough. Define a new fuzzy graph $\Gamma'$ where $\mu'(uv) = c$ and remain all other memberships unchanged. Then, $\mu'_u=\mu_u+p$ and $\mu'_v=\mu_v+p$ where  $\mu'_w=\mu_w$ for all other vertices $w \neq u,v$. Hence $\Delta = \SOG(\Gamma') - \SOG(\Gamma)$. Thus, 
\[
\begin{aligned}
\Delta&=p \sqrt{(\mu_u+p)^2+(\mu_v+p)^2} 
+\sum_{w \sim u} \mu(uw) \left[ \sqrt{(\mu_u+p)^2+\mu_w^2}-\sqrt{\mu_u^2+\mu_w^2} \right]+\\
&\qquad+\sum_{x \sim v} \mu(xv) \left[ \sqrt{(\mu_v+p)^2+\mu_x^2}-\sqrt{\mu_v^2+\mu_x^2} \right].
\end{aligned}
\]
By utilize the inequality $\sqrt{(a + c)^2 + b^2} - \sqrt{a^2 + b^2} > 0$ for $c > 0$. Then, the term of $p$ established that $p \sqrt{(\mu_u+p)^2+(\mu_v+p)^2}>0$. Since all terms are positive, $\Delta > 0$. It follows that $\SOG(\Gamma)$ is maximized only when the underlying graph is the complete graph $\KK_n$.

\case{3} If the underlying graph is $\KK_n$ and all edge memberships are equal to some constant $p \in (0,1]$, every vertex has fuzzy degree $\mu_u=(n-1)p$. Then, according to Lemma~\ref{lemstar01}, 
\begin{equation}~\label{eqq2MainThmSom3}
\SOG(\KK_n^f) = \sum_{uv \in E(\KK_n)} p \sqrt{((n-1)p)^2+((n-1)p)^2}
\end{equation}
Thus, from~\eqref{eqq2MainThmSom3} we obtain 
\begin{equation}~\label{eqq3MainThmSom3}
\SOG(\KK_n^f)=\binom{n}{2} \, p \cdot (n-1)p\sqrt{2}.
\end{equation}
Then, based on~\eqref{eqq3MainThmSom3}
\[
\SOG(\KK_n^f)=\frac{\sqrt{2}}{2} n (n-1)^2 p^2
\]
which the relationship~\eqref{eqq1MainThmSom3} holds.
\end{proof}
Since any non-uniform distribution or incomplete underlying graph yields a strictly smaller value, the uniform fuzzy complete graph $\KK_n^f$ is the unique maximizer.

\begin{corollary}~\label{corfuzzy003}
Among all a fuzzy complete graph $\KK_n^f$ with all edge memberships equal to some constant $p \in (0,1]$, 
\begin{equation}~\label{eqq0corfuz01}
\frac{\sqrt{2} \, m_{\mu}^2}{n-1} \leqslant \SOG(\Gamma) \leqslant \frac{\sqrt{2}}{2} n(n-1)^2 p_{\max}^2,
\end{equation}
where $p_{\max}$ is the maximum edge membership value in $\Gamma$. 
\end{corollary}
\begin{proof}
Assume $d_i=\mu_{v_i}$ be the fuzzy degrees, where $m_{\mu} = \frac{1}{2} \sum d_i$ for a fuzzy complete graph $\KK_n^f$. Then, 
\begin{equation}~\label{eqq1corfuz01}
\sum_{uv\in E(\Gamma)} \sqrt{d(u)^2+d(v)^2} \geqslant \sqrt{2} \cdot (n-1) m_{\mu},
\end{equation}
Thus, from~\eqref{eqq1corfuz01} we obtain 
\[
\SOG(\Gamma)\geqslant  \frac{1}{\mu(u,v)}\left(\sqrt{2} \cdot (n-1) m_{\mu}\right).
\]
Hence, for $m_{\mu}$ in $\KK_n^f$ the minimum occurs when the memberships are concentrated on as few edges as possible.

Since $\sqrt{d(u)^2+d(v)^2} \leqslant \sqrt{2} \max(d(u),d(v)) \leqslant \sqrt{2} \cdot \max d_i$, it follows that
\begin{equation}~\label{eqq2corfuz01}
  \SOG(\Gamma) \leqslant \binom{n}{2} \cdot p \cdot (n-1)p \sqrt{2},  
\end{equation}
with $p = 2m_{\mu}/[n(n-1)]$. Thus, from~\eqref{eqq1corfuz01} and \eqref{eqq2corfuz01} the relationship~\eqref{eqq0corfuz01} holds.
\end{proof}

After characterising the extremal fuzzy graphs for the fuzzy Sombor index ($\alpha=1$), we now examine its generalised form. The powering function $f \mapsto f^{\alpha}$ with $\alpha>1$ is strictly growing and convex on $[0,\infty)$, therefore the graph maximising the individual terms should also maximise the powered sum. The following result validates this expectation, demonstrating that the extremal graph continues to be the same as in the conventional situation.

\begin{theorem}~\label{MainThmSom4}
For $\alpha > 1$, among fuzzy graphs on $n$ vertices. Then, according to Theorem~\ref{MainThmSom3}, $\SOG_{\alpha}(\Gamma)$ maximum on $\KK_n^f$,
\begin{equation}~\label{eqq1MainThmSom4}
 \SOG_{\alpha}(\KK_n^f)=\binom{n}{2}\left(c \cdot (n-1)c \sqrt{2} \right)^{\alpha}.   
\end{equation}
\end{theorem}
\begin{proof}
Assume that $\nu(v)=1$ for all $v \in V(\Gamma)$. For each pair $u,v$, consider
\[
\lambda_{uv} =
\begin{cases}
\mu(uv)\sqrt{\mu_u^2 + \mu_v^2}, & \text{if } \mu(uv)>0,\\
0, & \text{otherwise}.
\end{cases}
\]
Then, it follows that 
\[
\SOG(\Gamma) = \sum_{u<v} \lambda_{uv}, \qquad
\SOG_{\alpha}(\Gamma) = \sum_{u<v} \lambda_{uv}^{\alpha}.
\]
If $\alpha=1$, according to Theorem~\ref{MainThmSom3}, we find that $\SOG(\Gamma)$ is maximized if and only if $\Gamma\cong \KK_n^f$ with constant edge membership $c$. Thus,  for all positive terms $\lambda_{uv}$ are equal, and a direct computation yields
\begin{equation}~\label{eqq2MainThmSom4}
\lambda_{uv} = c \sqrt{((n-1)c)^2 + ((n-1)c)^2}.
\end{equation}
In this case, if $u \neq v$. Then, from~\eqref{eqq2MainThmSom4} follows that $\lambda = \sqrt{2}(n-1)c^2$.

If $\alpha > 1$, the function $f(\eta)=\eta^{\alpha}$ is strictly increasing and strictly convex on $[0,\infty)$. Thus, if $\Gamma\cong \KK_n^f$, it follows that $\SOG_{\alpha}(\Gamma)$ is maximized uniquely at $\KK_n^f$ with uniform membership. Then, we will discuss the following cases:

\case{1} Assume that $\Gamma$ is not complete. Then, according to Theorem~\ref{MainThmSom3}, the positive membership strictly increases $\SOG(\Gamma)$. Since $f(\eta)$ is strictly increasing, this operation also strictly increases $\SOG_{\alpha}(\Gamma)$. Hence, any maximizer must be complete.

\case{2} If $\Gamma\cong \KK_n$ but the edge memberships are not all equal. Then the values $\lambda_{uv}$ are not constant. Since $f(\eta)$ is strictly convex. Then, 
\begin{equation}~\label{eqq3MainThmSom4}
\left( \frac{1}{\binom{n}{2}} \sum_{u<v} \lambda_{uv} \right)^{\alpha}<\frac{1}{\binom{n}{2}} \sum_{u<v} \lambda_{uv}^{\alpha}. 
\end{equation}
Thus, for a fixed value of $\lambda_{uv}$, according to~\eqref{eqq3MainThmSom4} the sum of $\lambda_{uv}^{\alpha}$ is maximized when all $\lambda_{uv}$ are equal.
Therefore, we conclude that $\SOG_{\alpha}(\Gamma)$ is maximized uniquely when $\Gamma\cong  \KK_n^f$ with constant membership $c$. Thus, by considering $\lambda=\sqrt{2}(n-1)c^2$ when $u\neq v$, it gives that the relationship~\eqref{eqq1MainThmSom4} holds.
\end{proof}

The maximising fuzzy graph's independence from the value of $\alpha > 1$ is a positive aspect of the generalised fuzzy Sombor index. This occurrence often occurs when a nonnegative quantity is treated to a convex transformation.
\section{Fuzzy Sombor Index on Fuzzy Trees and Unicyclic Fuzzy Graphs}

Trees and unicyclic graphs have been widely investigated in classical Sombor index theory, with the path $\PP_n$ and the star $\SN_n$ frequently appearing as the extremal structures for the minimum and maximum values, respectively. It is therefore natural to study if similar results hold in the fuzzy setting, where both the underlying topology and edge membership values might fluctuate continually.

This section describes the fuzzy trees and unicyclic graphs that extremize the fuzzy Sombor index $\SOG(\Gamma)$ under two constraints: fixed order $n$ and fixed fuzzy size $m_{\mu}$.

We begin with some fuzzy trees. A fuzzy tree is a connected fuzzy graph with an underlying crisp graph that is a tree (connected and acyclic).
Let $\TG$ be a fuzzy tree on $n \geq 2$ vertices, among Theorem~\ref{Thmfuzzytree01} we consider $\nu(v) = 1$ for all $v\in V(\TG)$ and fixed fuzzy size $m_{\mu} > 0$.
\begin{theorem}~\label{Thmfuzzytree01}
 The fuzzy Sombor index on $\TG$ satisfies
 \begin{equation}~\label{eqq1Thmfuzzytree01}
\SOG(\PP_n^f) \leqslant \SOG(\TG) \leqslant \SOG(\SN_n^f),
 \end{equation}
the lower bound is attained if and only if $\TG\cong \PP_n^f$ and the upper bound is attained if and only if $\TG\cong \SN_n^f$. 
\end{theorem}
\begin{proof}
By considering results in Lemma~\ref{lemstar01} and Theorem~\ref{MainThmSom2}. Assume that $\nu(v)=1$ for vertex $v\in V(\TG)$.

\medskip 

  {\bf Lower bound.} Given that every fuzzy tree is a cross-connected fuzzy graph, the fuzzy path will always be the minimiser within the fuzzy tree class. Then, if $\TG\cong \PP_n^f$, $\SOG(\PP_n^f) \leqslant \SOG(\TG)$. 

\medskip

  {\bf Upper bound.} Assume the edge $e=xy$ where both $x$ and $y$ have a fuzzy degree of at least $2$. Identify a leaf $u$ attached to $x$ and another leaf $w$ attached to $y$. By shipping a small positive membership $\varepsilon$ from the edge $xy$ to a new connection that focuses the weight at one central vertex, the term that includes a high aggregated degrees will be substituted by terms associated with a one vertex with considerably greater fuzzy degree.

Suppose that $\Delta(\TG)$ be the current maximum fuzzy degree. Thus,  for $\SN_n^f$ the center has fuzzy degree $(n-1)c=m_{\mu}$, and the pendent vertex have degree $c=m_{\mu}/(n-1)$. Then, 
\begin{equation}~\label{eqq2Thmfuzzytree01}
 c \sqrt{(m_{\mu})^2 + c^2} = c \sqrt{m_{\mu}^2 + \left(\frac{m_{\mu}}{n-1}\right)^2}.   
\end{equation}

This relationship~\eqref{eqq2Thmfuzzytree01} is maximized by the star among all non-star trees where concentrating the total fuzzy size $m_\mu$ on a single biggest degree vertex increases the sum more than spreading it across several vertices. Equivalently, the star degree sequence $(n-1,1,\dots,1)$ majorizes every other tree degree sequence. Thus,  $\SOG(\TG) \leqslant \SOG(\SN_n^f)$. Then, if $\TG\cong \SN_n^f$, equality~\eqref{eqq1Thmfuzzytree01} holds precisely for the fuzzy star, and then the maximum is attained when all $n-1$ star edges have equal membership $c=m_\mu/(n-1)$.
\end{proof}

A \emph{fuzzy unicyclic graph} is a connected fuzzy graph whose underlying crisp graph contains exactly one cycle.

\begin{theorem}~\label{Thmfuzzytree02}
Among all fuzzy unicyclic graphs on $n \geq 3$ vertices with fixed fuzzy size $m_{\mu}$, the fuzzy Sombor index $\SOG(\Gamma)$ is
\begin{itemize}
    \item maximized when the underlying graph is a cycle with a star; 
    \item minimized when the underlying graph is the cycle $C_n$, with edge memberships distributed precisely.
\end{itemize}
\end{theorem}
\begin{proof}
Since underlying crisp graph $G$ is connected and contains exactly one cycle. We prove the minimum and maximum separately.

\medskip

 \textbf{Minimum.} Any fuzzy unicyclic graph that is not a pure fuzzy cycle can be shifted into one with a strictly reduced fuzzy Sombor index, conserving interaction, unicyclicity, and the fuzzy size $m_{\mu}$.
 
Assume $G$ has at least one pendant tree connected to the cycle. Let $v$ be a vertex on the cycle with a pendant path or tree connected. Then, by determining a leaf $u$ in this pendant structure and consider shifting a minor positive membership value $\varepsilon > 0$ from an edge in the pendant tree closer to $N(u)$ toward an edge on the cycle itself. Similarly, for a pendent vertex $v$ where $d(u)\neq d(v)$ and $N(u)\neq N(v)$. Since pendant constructions yield vertices with comparatively small fuzzy degrees at the leaves, shifting membership out of them and along the cycle tends to reconcile the degrees more precisely. 

Once the underlying graph is $C_n$, the fuzzy Sombor index reduces to a function of the edge memberships $\eta_1, \eta_2, \dots, \eta_n$ around the cycle with $\sum \eta_i=m_{\mu}$. Therefore, by considering the cyclic symmetry and the convexity properties of the map $\eta \mapsto \eta \sqrt{x^2+y^2}$ where adjacent degrees are linear combinations of consecutive $x_i$. In this case, this function is minimized when all $\eta_i$ are equal. 
Thus, the fuzzy Sombor index $\SOG(\Gamma)$ is minimized when $\Gamma\cong C_n^f$ with uniform edge membership $m_{\mu}/n$.

\textbf{Maximum.} To compare unicyclic fuzzy graphs with $n$ vertices, we use a \emph{concentration transformation}. Assume that $v$ is a vertex with the maximal fuzzy degree $\mu_v = \max_{u \in V(\Gamma)} \mu_u.$
Hypothesise that edge memberships are dispersed over numerous vertices or along pendant paths.

Let $m_1,m_2$ be a tow edges such that $m_2$ is incident to $v$, while $m_1$ is not. Thus, sufficiently small $\varepsilon > 0$, decrease the membership of $m_1$ by $\varepsilon$ and increase the membership of $m_2$ by $\varepsilon$, adjusting the structure if needed so that the graph remains connected and unicyclic.
Thus, let be established the following function $f(x,y) = \sqrt{x^2 + y^2}$ where is strictly increasing and convex in each argument for $x,y > 0$. Increasing a larger degree $\mu_v$ results in an upper degree gain compared to reducing lower degrees. Thus, for sufficiently tiny $\varepsilon$, any transformation strictly raises $\SOG(\Gamma)$.

Other unicyclic configurations distribute degrees more uniformly, resulting in a lower value of $\SOG(\Gamma)$ according to the convexity argument mentioned above.
\end{proof}

These findings apply the conventional extremal behaviour of the Sombor index to the fuzzy situation. Path-like (or cycle-like) patterns in trees and unicyclic graphs reduce the index by promoting degree uniformity, whereas star-like structures maximise it by forming one vertex with a very high fuzzy degree. Continuous edge memberships allow for finer optimisation than in the crisp case, but the underlying topological structure remains the determining factor.

\begin{table}[H]
\centering
\begin{tabular}{|c|cccc|cccc|}
\hline
\multirow{2}{*}{$n$} & \multicolumn{4}{c|}{$m_{\mu}=0.5$} & \multicolumn{4}{c|}{$m_{\mu}=1$} \\
\cline{2-9}
 & Path & Star & Cycle & Complete & Path & Star & Cycle & Complete \\
\hline
6  & 0.2144 & 0.2550 & 0.1179 & 0.1179 & 0.8577 & 1.0198 & 0.4714 & 0.4714 \\
8  & 0.1671 & 0.2525 & 0.0884 & 0.0884 & 0.6685 & 1.0102 & 0.3536 & 0.3536 \\
10 & 0.1360 & 0.2515 & 0.0707 & 0.0707 & 0.5441 & 1.0062 & 0.2828 & 0.2828 \\
12 & 0.1144 & 0.2510 & 0.0589 & 0.0589 & 0.4577 & 1.0041 & 0.2357 & 0.2357 \\
15 & 0.0923 & 0.2506 & 0.0471 & 0.0471 & 0.3692 & 1.0025 & 0.1886 & 0.1886 \\
20 & 0.0697 & 0.2503 & 0.0354 & 0.0354 & 0.2788 & 1.0014 & 0.1414 & 0.1414 \\
25 & 0.0560 & 0.2502 & 0.0283 & 0.0283 & 0.2238 & 1.0009 & 0.1131 & 0.1131 \\
30 & 0.0467 & 0.2501 & 0.0236 & 0.0236 & 0.1869 & 1.0006 & 0.0943 & 0.0943 \\
40 & 0.0351 & 0.2501 & 0.0177 & 0.0177 & 0.1405 & 1.0003 & 0.0707 & 0.0707 \\
50 & 0.0281 & 0.2501 & 0.0141 & 0.0141 & 0.1126 & 1.0002 & 0.0566 & 0.0566 \\
\hline
\end{tabular}
\caption{Comparing adjusted fuzzy Sombor indices with corresponding graph invariants throughout graph families with $m_{\mu}=1$ and $m_{\mu}=0.5$.}
\label{tab01fuzzy}
\end{table}
Table~\ref{tab01fuzzy} displays the adjusted Sombor indices for fundamental graph families. For small $n\leqslant 5$, the Path graph has the maximum value, while the other graphs have identical minimum values.  The growth of $n$ emphasize that $\SOG$ lowers in all families.

The behaviour of the fuzzy Sombor index across  fundamental families of fuzzy graphs with fixed fuzzy size $m_{\mu}=1$ depicts among Figure~\ref{fig01fuzzmu1}. 
In contrast, the fuzzy path achieves the lowest values, indicating a balanced and dispersed structure as shown in Table~\ref{tab01fuzzy}. 
The fuzzy cycle falls somewhere between these two extremes. 

\begin{figure}[H]
    \centering
    \includegraphics[width=0.7\linewidth]{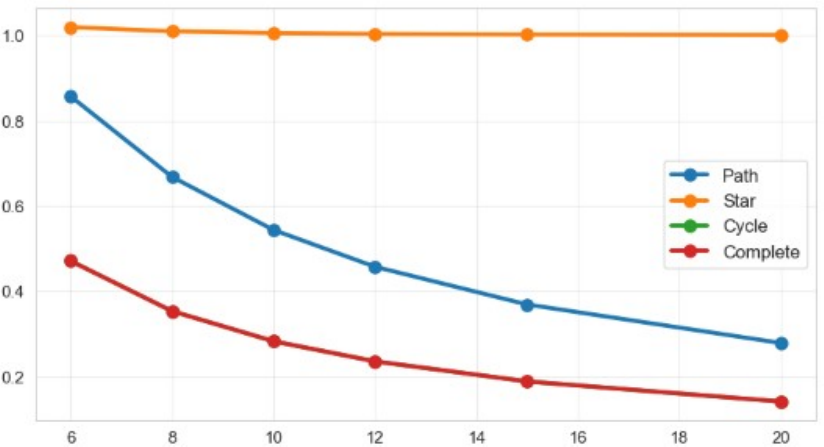}
    \caption{Fuzzy Sombor index in families of graph with $m_{\mu}=1$ where Cycle and Complete become very close.}
    \label{fig01fuzzmu1}
\end{figure}

Figure~\ref{fig02fuzzmu1} demonstrate that the behavior of the fuzzy Sombor index across  fundamental families of fuzzy graphs based on the value of Table~\ref{tab01fuzzy} with fixed fuzzy size $m_{\mu}=0.5$. 
\begin{figure}[H]
    \centering
    \includegraphics[width=0.8\linewidth]{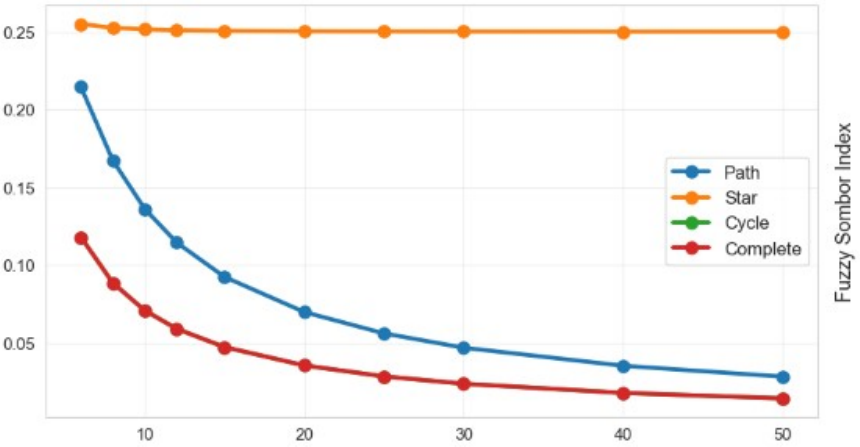}
    \caption{Compare value of Fuzzy Sombor index in families of graph with $m_{\mu}=0.5$.}
    \label{fig02fuzzmu1}
\end{figure}

\section{Relations and Inequalities with Other Fuzzy Topological Indices}

In this section, the fuzzy Sombor index incorporates details regarding edge strengths and vertex degrees in a nonlinear Euclidean algorithm. It is obvious to investigate its link to other well-known fuzzy topological indices. For establishing plenty sharp inequalities between $\SOG(\Gamma)$ and classical fuzzy topological indices. Consider $\nu(v) = 1$ for all vertices unless otherwise specified. 

By recall Definition~\ref{fuzzrand001}, Proposition~\ref{fuzzyprorandic} established the lower bound of $\SOG$ based on $R(\Gamma)$.
\begin{proposition}~\label{fuzzyprorandic}
If all fuzzy degrees are positive,
\[
\SOG(\Gamma) \geq \sqrt{2} \, R(\Gamma) \cdot \delta_{\mu},
\]
where $\delta_{\mu}$ is the minimum positive fuzzy degree. 
\end{proposition}

\begin{theorem}~\label{ZagrebThm01}
For any regular fuzzy graph $\Gamma$,
\begin{equation}~\label{eqq1ZagrebThm01}
\SOG(\Gamma) \leqslant \sqrt{2} \left(\sqrt{(2n-2)\,M_1(\Gamma)} \cdot m_{\mu}\,+m(n-1)\right).
\end{equation}
\end{theorem}
\begin{proof}
Recall the fuzzy Zagreb's index according to Definition~\ref{ZagrebIndices001}, and assume $\Gamma$ be a fuzzy regular graph on $n\geqslant 3$ vertices, where 
\begin{equation}~\label{eqq2ZagrebThm01}
\SOG(\Gamma)\leqslant \sqrt{ \sum_{uv} \mu(uv)^2 \cdot (\mu_u^2+\mu_v^2)}.
\end{equation}
Let $\mu_{\max}=\max\{\mu_u,\mu_v\}$ and $\mu_{\min}=\min\{\mu_u,\mu_v\}$. Then, for the maximum term satisfy $\sqrt{\mu_u^2 + \mu_v^2} \leqslant \sqrt{2} \mu_{\max}$ and for the minium term yields $\sqrt{\mu_u^2 + \mu_v^2} \geqslant \sqrt{2} \mu_{\min}.$ Thus, $\SOG_{\max}(\Gamma) \geqslant \sqrt{2} \mu_{\min}$ and $\SOG_{\max}(\Gamma) \leqslant \sqrt{2} \mu_{\max}$. 

According to Theorem~\ref{Thmper001}, if $\SOG(\Gamma)\leqslant \sqrt{2}m(n-1)$. It observers that $\SOG(\Gamma) \leqslant \sqrt{2} \max_u \mu_u \cdot m_{\mu}$. Then, $\SOG_{\max}(\Gamma) \leqslant \sqrt{2} m_\mu\, \mu_{\max}$. Thus, the sharp bound of $\SOG$ yields that 
\begin{equation}~\label{eqq3ZagrebThm01}
\SOG(\Gamma)^2 \leqslant 2 m_{\mu} \sum_{uv} \left(\mu(uv) \left(\mu_{\max}^2+\mu_{\min}^2\right)\sqrt{\max(\mu_u^2, \mu_v^2)}\right) \leqslant 2\,m_{\mu} M_1(\Gamma)\,\max(\mu_u^2, \mu_v^2),
\end{equation}
where 
\[
\sum_{uv} \mu(uv) \sqrt{\mu_u^2 + \mu_v^2} \leqslant \sqrt{ \left( \sum_{uv} \mu(uv) \right) \left( \sum_{uv} \mu(uv) (\mu_u^2 + \mu_v^2) \right) }.
\]
Therefore, based on~\eqref{eqq3ZagrebThm01} and when $\SOG(\Gamma)\leqslant\sqrt{(2n-2)\,M_1(\Gamma)}$ we obtain 
\begin{equation}~\label{eqq4ZagrebThm01}
\SOG(\Gamma) \leqslant \sqrt{2} \left(\sqrt{(2n-2)\,M_1(\Gamma)} \cdot m_{\mu}\right).
\end{equation}
 Thus, from~\eqref{eqq4ZagrebThm01} the relationship~\eqref{eqq1ZagrebThm01} hold.
\end{proof}
In the previous theorem, the first Zagreb index for a fuzzy graph was considered in the regular case, that is, when all vertex degrees are equal. In this situation, the maximum and minimum degrees coincide. We now extend these results to the case where the fuzzy graph is irregular. In this case, we find that the sharp bound of this index satisfies the following:
\begin{equation}~\label{eqq5ZagrebThm01}
\SOG(\Gamma) \leqslant 2 m_{\mu} \sum_{uv} \left(\mu(uv) \left(\mu_{\max}^2+\mu_{\min}^2\right)\sqrt{(2n-2)M_1(\Gamma)}\right).
\end{equation}

\begin{theorem}~\label{normThm001}
Let $\Gamma$ be a fuzzy graph on $n\geqslant 3$ vertices with maximum degree $\Delta(\Gamma)=\max\{\mu_u+\mu_v\},$ where $uv\in E(\Gamma)$ and $\mu_u\neq \mu_v$. Then, 
\begin{equation}~\label{eqq1normThm001}
\SOG(\Gamma) \leqslant \frac{1}{\sqrt{2}}\, m_{\mu}\,\Delta(\Gamma)
+ \sqrt{\sum_{uv\in E(\Gamma)} 
\frac{(\mu_u+\mu_v)\,\max\{\mu_u^2,\mu_v^2\}}
{2\,\mu_u\,\mu_v\,\min\{\mu_u^2,\mu_v^2\}} }.
\end{equation}
if and only if 
\[
\SOG(\Gamma) \leqslant \frac{\Delta_{\mu}(\Gamma)+\delta_{\mu}(\Gamma)}{\sqrt{2}} \, N(\Gamma),
\]
where $\Delta_{\mu}(\Gamma) = \max_u \mu_u$ and $\delta_{\mu}(\Gamma) = \min_u \mu_u$.
\end{theorem}
\begin{proof}
Recall the relationship of $N(\Gamma)$ from Definition~\ref{deffNirmala001} and according to Theorem~\ref{ZagrebThm01}, we find that if $\Gamma$ is regular fuzzy graph, $\SOG(\Gamma) \leqslant \sqrt{2} \max_u \mu_u \cdot m_{\mu}$. Thus, if $uv\in E(\Gamma)$ and $\mu_u=\mu_v$, it observe that $\mu_u+\mu_v\geqslant \sqrt{2\,(\mu_u^2+\mu_v^2)}$. Therefore, 
\begin{equation}~\label{eqq2normThm001}
\SOG(\Gamma)\leqslant \frac{1}{\sqrt{2}} \sum_{uv} \mu(uv) (\mu_u+\mu_v)\leqslant \frac{1}{\sqrt{2}} \, m_{\mu} \cdot \max_{uv} (\mu_u + \mu_v)
\end{equation}
Hence, based on~\eqref{eqq2normThm001} and by considering $m_{\mu}$, yields 

\begin{align*}
\mu_{\max}\,\sqrt{\mu_u^2+\mu_v^2}
&\leqslant \SOG(\Gamma)
\leqslant \mu_{\max}\,(\mu_u^2+\mu_v^2),\\
\Delta(\Gamma)\,\sqrt{\mu_u^2+\mu_v^2}
&\leqslant \SOG(\Gamma)
\leqslant \Delta(\Gamma)\,(\mu_u^2+\mu_v^2).
\end{align*}

It follows that $$\SOG(\Gamma)\leqslant \sqrt{\Delta(\Gamma)\,(\mu_u^2+\mu_v^2)+(\mu_u+\mu_v)\,\max\{\mu_u^2,\mu_v^2\}}.$$ 
Since $\mu_u + \mu_v \leq 2 \Delta_{\mu}$, the relationship~\eqref{eqq2normThm001} implies that 
\begin{equation}~\label{eqq3normThm001}
\SOG(\Gamma)\leqslant \sqrt{\sum_{uv\in E(\Gamma)} 
\frac{m_{\mu}\Delta(\Gamma)(\mu_u+\mu_v)\,\max\{\mu_u^2,\mu_v^2\}}
{2\,\mu_u\,\mu_v\,\min\{\mu_u^2,\mu_v^2\}} }.
\end{equation}
Therefore, it established the sharp bound of $\SOG$, 
\begin{equation}~\label{eqq4normThm001}
\SOG(\Gamma) \leqslant \frac{\Delta_{\mu}}{\sqrt{2}} \sum_{uv} \mu(uv) \cdot \frac{\mu_u + \mu_v}{\Delta_{\mu}} \leqslant \frac{\Delta_{\mu}}{\sqrt{2}} N(\Gamma),
\end{equation}
Thus, from~\eqref{eqq3normThm001} and \eqref{eqq4normThm001} the relationship~\eqref{eqq1normThm001} holds.
\end{proof}
Figure ~\ref{fig03fuzzmu1} compares four significant fuzzy topological indices, where the visualisation demonstrates the unique behaviour of each index. 

\begin{figure}[H]
    \centering
    \includegraphics[width=1\linewidth]{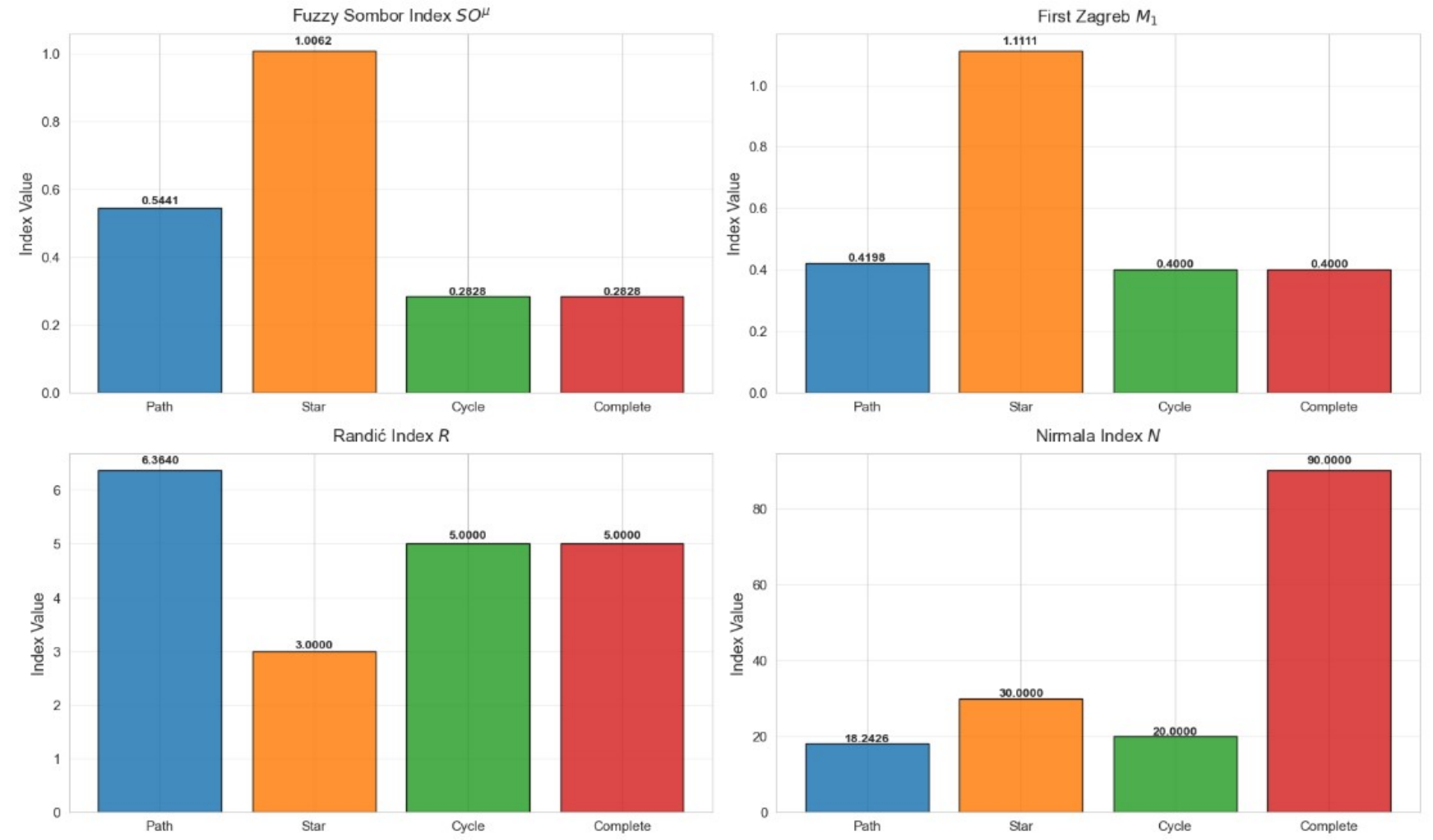}
    \caption{The Fuzzy Sombor Index, along with other essential indices, includes $M_1(\Gamma)$, $R(\Gamma)$ and $N(\Gamma)$.}
    \label{fig03fuzzmu1}
\end{figure}
The value under the identical settings, with the fuzzy Sombor index responding more sensitively to structural centralisation than the other indices.
The results among Table~\ref{tab02fuzzy} reveal that the Star graph has the maximum Sombor index, whilst the Path graph has the maximum values in the other three indices. Both the Cycle and Complete graphs produce identical results for $\SOG$, $M_1$, and $R$, as expected given their structural similarities in this sequence.
\begin{table}[H]
\centering
\begin{tabular}{|l|cccc|}
\hline
Graph     & $\SOG$ & $M_1$   & $R$   & $N$     \\
\hline
Path      & 0.5441    & 0.4198     & 6.364    & 18.2426    \\\hline
Star      & 1.0062    & 1.1111     & 3.000    & 30.0000    \\\hline
Cycle     & 0.2828    & 0.4000     & 5.000    & 20.0000    \\\hline
Complete  & 0.2828    & 0.4000     & 5.000    & 90.0000    \\
\hline
\end{tabular}
\caption{An essential indices for $n=10$, $m_{\mu}=1.0$.}
\label{tab02fuzzy}
\end{table}

To extended the study on $\SOG$, we present the impact $\SOG$  on extremal fuzzy trees. 
The examination of extremal fuzzy trees for the fuzzy Sombor index displays a distinct fundamental influence that parallels rather refines the classical crisp instance. In the crisp adjusting, the path $\PP_n$ minimises and the star $\SN_n$ maximises the majority of degree-based indices, especially the Sombor index. In the fuzzy framework, the same topological structures remain extremal, but continuous edge memberships $c \in (0,1]$ add suppleness.

The fuzzy path $\PP_n^f$ with practically identical edge memberships obtains the smallest value $\mu_{\min}$ by preserving the most symmetrical fuzzy degree sequence attainable under the tree constraint. The terms $\mu(uv)\sqrt{\mu_{\max}^2+\mu_{\min}^2}$ remains minimal since neighbouring vertices have similar fuzzy degrees $\delta_{\mu}$. The aforementioned setup is particularly effective in modelling linear or chain-like uncertain networks, when the objective is to minimise "structural intensity," as assessed by the Sombor index. Also, the fuzzy star $\SN_n^f$ with universal edge membership $c=m_{\mu}/(n-1)$ focuses the full fuzzy size on edges incident $m_{\mu}$ to a single center vertex. The findings in an enormous fuzzy degree $m_{\mu}$ coupled with many little degrees $c$, maximising every element $\sqrt{\mu_c^2+c^2}$ due to the quadratic increase of the Euclidean term. As a result, the fuzzy Sombor index grows roughly as $m_{\mu}^2$, illustrating the index's sensitivity to hub-like structures. In behaviour, this implies that the fuzzy Sombor index may efficiently detect or rate centralisation and vulnerability in fuzzy networks.

\section{Conclusion}
This paper presents a thorough study of the extremal characteristics of the fuzzy Sombor index and its general form $\SOG_{\alpha}(\Gamma)$ for $\alpha>1$. We produced strong lower and upper bounds on $n$ vertices, demonstrating that the fuzzy path (with balanced memberships) reduces the index among connected fuzzy graphs, whereas the uniform fuzzy full graph increases it. These findings have been improved for the significant families of fuzzy trees and fuzzy unicyclic graphs, where the star and tadpole structures appear as maximizes and paths/cycles as minimizes.

The inequalities derived between the fuzzy Sombor index and other fuzzy topological indices (Zagreb, Randić, Nirmala), emphasising its intermediate position between linear and quadratic measures. Edge memberships are continuous, which allows for finer optimisation than in the crisp case. Then, $\SOG$ performs well as a sensitive descriptor of both degree variance and edge strength concentration.




\begin{thebibliography}{9}

\bibitem{Geometric2021approach} I.~Gutman, Geometric approach to degree-based topological indices: Sombor indices, {\it MATCH~Commun.~Math.~Comput.~Chem.} \textbf{86} (1) (2021) 11--16.

\bibitem{Cruz2021Gutman} R.~Cruz, I.~Gutman, \& J.~Rada, Sombor index of chemical graphs, {\it Appl.~Math.~Comp.} \textbf{399}(126018) (2021), \href{https://doi.org/10.1016/j.amc.2021.126018}{DOI: 10.1016/j.amc.2021.126018}.

\bibitem{Cruz2021Rada} R.~Cruz, \& J.~Rada, Extremal values of the Sombor index in unicyclic and bicyclic graphs, {\it J.~Math.~Chem.} \textbf{59} (2021), 1098--1116, \href{https://doi.org/10.1007/s10910-021-01232-8}{DOI: 10.1007/s10910-021-01232-8}.

\bibitem{cruz2020extremal} R.~Cruz, J.~Monsalve, \& J.~Rada, Extremal values of vertex-degree-based topological indices of chemical trees, {\it Appl.~Math.~Comput.} \textbf{380} (2020), 125281, \href{https://doi.org/10.1016/j.amc.2020.125281}{DOI: 10.1016/j.amc.2020.125281}.

\bibitem{cruz2021threependent}
R.~Cruz, J.~Rada, \& J.~M.~Sigarreta, Sombor index of trees with at most three branch vertices, {\it Appl.~Math.~Comput.} \textbf{409}(126414) (2021), \href{https://doi.org/10.1016/j.amc.2021.126414}{DOI: 10.1016/j.amc.2021.126414}. 


\bibitem{Das2021Cevik} K.~C.~Das, A.~S.~Cevik, I.~N.~Cangul, \& Y.~Shang, On Sombor Index, 
\emph{Symmetry},  \textbf{13}(1) (2021), \href{https://doi.org/10.3390/sym13010140}{DOI: 10.3390/sym13010140}. 

\bibitem{Das2021Shang} K.~C.~Das, \& Y.~Shang, Some Extremal Graphs with Respect to Sombor Index, {\it Mathematics.} \textbf{9} (2021), \href{https://doi.org/10.3390/math9111202}{DOI: 10.3390/math9111202}.


\bibitem{Todeschini2000} R.~Todeschini, \& V.~Consonni, Handbook of Molecular Descriptors, Wiley–VCH: Weinheim, Germany, 2000.

\bibitem{Some2025Pal} B.~Some, \& A.~Pal, Sombor index of fuzzy graph, {\it TWMS~J.~App.~and Eng.~Math.} \textbf{15}(6) (2025), 1325--1346.

\bibitem{Poulikrad2024} S.~Poulik, G.~Ghorai, \& Q.~Xin, Explication of crossroads order based on Randic index of graph with fuzzy information,  {\it Soft Computing.} \textbf{28}(3) (2024), 1851--1864, \href{https://doi.org/10.1007/s00500-023-09453-6}{DOI: 10.1007/s00500-023-09453-6}.

\bibitem{Kalathian2020S} S.~Kalathian, S.~Ramalingam, S.~Raman, \& N.~Srinivasan, Some topological indices in fuzzy graphs, {\it J.~Intell.~Fuzzy~Syst.} \textbf{39} (2020)6033--6046, \href{http://dx.doi.org/10.3233/JIFS-189077}{DOI: 10.3233/JIFS-189077}. 

\bibitem{sunitha1999fuzzytrees}
M.~S.~Sunitha, \& A.~Vijayakumar, A characterization of fuzzy trees,{\it Inf.~Sci.}  \textbf{113} (3--4) (1999), 293--300, \href{https://doi.org/10.1016/S0020-0255(98)00135-3}{DOI: 10.1016/S0020-0255(98)00135-3}.


\bibitem{Akram2023fuzzy} M.~Akram, A.~Habib, \& J.~C.~R.~Alcantud, Fuzzy topological indices with application to cybercrime problem, \emph{Granular Computing}. \textbf{8} (2023),  1549--1569, \href{https://doi.org/10.1007/s41066-023-00365-2}{DOI: 10.1007/s41066-023-00365-2}.

\bibitem{Kosari2023some}
S.~Kosari, Y.~Rao, H.~Zhang, \& A.~Rezaei, Some indices of picture fuzzy graphs and their applications, \emph{Computational and Applied Mathematics}. \textbf{42} (281) (2023), \href{https://doi.org/10.1007/s40314-023-02393-9}{DOI: 10.1007/s40314-023-02393-9}.




\end{thebibliography}
\end{document}